\documentclass{amsart}

\usepackage{graphics}
\usepackage{amsmath,amsfonts,latexsym,amssymb}
\usepackage{comment}
\usepackage{xcolor}
\usepackage{fullpage}

\usepackage[english]{babel}
\usepackage{amsthm}
\usepackage{mathrsfs}
\usepackage{pgf,tikz}
\usepackage{ulem}
\usepackage{amstext} 
\usepackage{array}   
\newcolumntype{C}{>{$}c<{$}} 
\definecolor{uququq}{rgb}{0.25,0.25,0.25}
\newtheorem{thm}{Theorem}[section]

\newtheorem{lem}[thm]{Lemma}
\newtheorem{prop}[thm]{Proposition}

\usetikzlibrary{arrows}
\theoremstyle{definition}
\newtheorem{defn}[thm]{Definition}
\theoremstyle{definition}
\newtheorem{rem}[thm]{Remark}
\theoremstyle{definition}
\newtheorem{ex}[thm]{Example}

\newcommand{\Z}{\mathbb{Z}}

\def\H{\mathrm{H}}

\def\M{\mathrm{M}}

\begin{document}

\title{Shiftable Heffter Spaces}

\author{M. Buratti}
\address{SBAI - Sapienza Universit\`a di Roma, Via Antonio Scarpa 16, I-00161 Roma, Italy}
\email{marco.buratti@uniroma1.it}

\author{A. Pasotti}
\address{DICATAM - Sez. Matematica, Universit\`a degli Studi di Brescia, Via
Branze 43, I-25123 Brescia, Italy}
\email{anita.pasotti@unibs.it}

\subjclass[2010]{}

\begin{abstract}
The shiftable Heffter arrays are naturally generalized to the shiftable Heffter spaces.
We present a recursive construction which starting from a single shiftable Heffter space leads to infinitely many other shiftable Heffter spaces of the same degree.
We also present a direct construction making use of pandiagonal magic squares leading to a shiftable $(16\ell^2,4\ell;3)$ Heffter space for any $\ell\geq1$.
Combining these constructions we obtain a shiftable $(16\ell^2mn,4\ell n;3)$ Heffter space for every triple of positive integers $(\ell,m,n)$ with $m\geq n$.
\end{abstract}

\maketitle

\section{Introduction}
Throughout this note, given two integers $a$ and $b$ we denote by $[a,b]$ the set of all integers $n$ such that $a\leq n\leq b$.
If $A$ is a set of integers, $A^+$ and $A^-$ will denote the sets of positive and negative integers of $A$, respectively.

A half-set of an abelian group $G$ is a subset $V$ of $G$ such that $V$ and $-V$ form a partition of $G\setminus\{0\}$.
Similarly, given a positive integer $v$, a half-set of $[-v,v]$ is a subset $V$ of $[-v,v]$ such that $V$ and $-V$ form a partition of $[-v,v]\setminus\{0\}$.
Clearly, a half-set of $[-v,v]$ can be also viewed as a half-set of $\Z_{2v+1}$.

We recall that a {\it partial linear space} is a point-block incidence structure $(V,\mathcal B)$ such that every pair of distinct points is contained in at most one block.
A partial linear space $(V,\mathcal B)$ is {\it resolvable} if ${\mathcal B}$ can be partitioned into {\it parallel classes} each of which is, in turn, a partition of $V$.
A $(v_r,b_k)$ {\it configuration} is a partial linear space with $v$ points each of degree $r$ and $b$ blocks each of size $k$. 
It is easy to see that the parameters of such a configuration satisfies the identity $vr=bk$.
For some results on resolvable configurations we refer to \cite{BS22,G}. 

Speaking of a $(v,k;r)$ space we will mean a resolvable $(v_r,b_k)$ configuration with $b=\frac{vr}{k}$.
A $(k^2,k;r)$ space is usually referred to as a $(r,k)$-{\it net}.
A $(v,k;r)$ space where every two distinct points are contained in {\it exactly} one block is a resolvable {\it Steiner $2$-design} S$(2,k,v)$.
In this case we necessarily have $r=\frac{v-1}{k-1}$, hence $b=\frac{v(v-1)}{k(k-1)}$.
For general background on nets and Steiner 2-designs we refer to \cite{BJL,CD}.

Let us say that two partitions $\mathcal P$ and $\mathcal Q$ of a set $V$ are {\it orthogonal} if any part of $\mathcal P$ 
has at most one element in common with every part of $\mathcal Q$. Using this terminology we can say that a $(v,k;r)$ space
is equivalent to a set ({\it resolution}) of $r$ mutually orthogonal partitions of a $v$-set into subsets of size $k$.
In particular, a $(v,k;1)$ space is nothing but a partition of a set of size $v$ into subsets of size $k$.

Although it is unusual, trivial and maybe even unaesthetic, it will be convenient to consider spaces with block size $k=1$. A $(v,1;r)$ space with point set $V$ is unique; its 
resolution is the partition of $V$ into singletons repeated $r$ times. This is the only case where repeated blocks are possible.

Inspired by the seminal paper by Caggegi, Falcone and Pavone \cite{CFP}, we say that a point-block incidence structure 
$(V,{\mathcal B})$ is {\it $G$-additive} if $V$ is a subset of an abelian group $G$
and each block is zero-sum in $G$. 
For an extensive literature on additive 2-designs and some generalizations see \cite{BMN,BN1,BN2,CFP2,FP,Pavone}.

We recently introduced  {\it Heffter spaces} \cite{BP1} as follows.
\begin{defn}
A $(v,k;r)$ {\it Heffter space} over an abelian group $G$ is a $G$-additive $(v,k;r)$ space whose points fill a half-set of 
$G$. 
\end{defn}

So a $(v,k;1)$ Heffter space is a partition of a half-set $V$ of an abelian group $G$ of order $2v+1$ into zero-sum $k$-subsets of $G$
and it is called a $(v,k)$ {\it Heffter system} on $V$.
The existence of a cyclic $(v,k)$ Heffter system for any admissible $v$ is implicitly deducible from the existence of a cyclic 
$k$-cycle decomposition of $K_{2kn+1}$ for any positive $n$ obtained, independently, in \cite{BGL} and \cite{BurDel}. 

The orthogonality of two $(v,k)$ Heffter systems 
${\mathcal P}=\{\mathcal P_1,\dots,\mathcal P_n\}$ and ${\mathcal Q}=\{\mathcal Q_1,\dots,\mathcal Q_n\}$ 
can be displayed by means of the $n\times n$ matrix $A$ whose entry $a_{i,j}$ is the intersection between $\mathcal P_i$ and $\mathcal Q_j$.
This matrix is said to be a {\it Heffter $(n;k)$ array} or, briefly, a $H(n;k)$. Heffter arrays were introduced by Archdeacon in \cite{A} and 
we refer to \cite{PD} for an extensive survey on them. Here, we just recall that an $\H(n; k)$ has been proved to exist if and only if $n \geq k \geq  3$
(see \cite{ADDY,CDDY,DW}).

A $(v,k;r)$ Heffter space is essentially a set of $r$ mutually orthogonal $(v,k)$ Heffter systems on the same half-set $V$.
Among the motivations for studying Heffter spaces we recall their connection with {\it orthogonal cycle systems}, an interesting topic
which received some attention recently \cite{BCP,KY}.
We point out that pairs of orthogonal cyclic cycle systems have been obtained in \cite{B,BDCY,CMPP,MT} as a consequence of constructive
results on Heffter arrays.
In \cite {BP1} and \cite{BP2} we have proved that for any pair of positive integers $(k,r)$ with $k\geq3$ there are infinitely many values of $v$ for which a $(v,k;r)$ Heffter space exists.
On the other hand, for the time being, the general existence problem for $(v,k;r)$ Heffter spaces does not appear to be easy.
In particular, considering that the above results have been obtained in suitable half-sets of elementary abelian groups,
we do not have $(v,k;r)$ Heffter spaces with $r>2$ where $2v+1$ is not a prime power.

We propose the following definition which is the natural generalization of the notion of an {\it integer Heffter array}.

\begin{defn}
An {\it integer} $(v,k;r)$ {\it Heffter space} is a $\Z$-additive $(v,k;r)$ space whose points fill a half-set of $[v,-v]$. 
It is said {\it shiftable} if $k$ is even and each block has exactly $\frac{k}{2}$ positive elements. 
\end{defn}

In the following, speaking of a {\it shiftable} Heffter space it will be tacitly understood that it is integer.

It is evident that every integer $(v,k;r)$ Heffter space can be viewed also as a $(v,k;r)$ Heffter space
on a half-set of $\Z_{2v+1}$.

The following two examples have been obtained by computer.

\begin{ex}\label{20}
Let $V$ be the half-set of $[-20,20]$ defined by 
$$V^+=\{2, 4, 6, 9, 10, 11, 12, 15, 17, 19\},\quad V^-=-\{1, 3, 5, 7, 8, 13, 14, 16, 18, 20\}$$
and consider the following partitions of $V$:

\medskip\noindent
\begin{center}
\begin{tabular}{|l|c|r|c|r|c|r|c|r|c|r|c|r|}
\hline {\quad\quad\quad$\mathcal {P}_1$}    \\
\hline $\{-1, 2, 17, -18\}$\\
\hline $\{-3, 6, 10, -13\}$\\
\hline $\{4, -5, -8, 9\}$ \\
\hline $\{-7, 12, 15, -20\}$  \\
\hline $\{ 11, -14, -16, 19\}$  \\
\hline
\end{tabular}\quad\quad\quad
\begin{tabular}{|l|c|r|c|r|c|r|c|r|c|r|c|r|}
\hline {\quad\quad\quad$\mathcal {P}_2$}    \\
\hline $\{-1,6,9, -14\}$\\
\hline $\{2, -5, -7, 10\}$  \\
\hline $\{-3, 4, 19, -20\}$\\
\hline $\{-8, 11, 15, -18\}$ \\
\hline $\{12, -13, -16, 17\}$  \\
\hline
\end{tabular}\quad\quad\quad
\begin{tabular}{|l|c|r|c|r|c|r|c|r|c|r|c|r|}
\hline {\quad\quad\quad$\mathcal {P}_3$}    \\
\hline $\{-1, 10, 11, -20\}$\\
\hline $\{2, -8, -13, 19\}$  \\
\hline $\{-3, 9, 12, -18\}$ \\
\hline $\{4, -7, -14, 17\}$\\
\hline $\{-5, 6, 15, -16\}$  \\
\hline
\end{tabular}
\end{center}

\medskip
It is easy to check that $\{\mathcal {P}_1,\mathcal {P}_2,\mathcal {P}_3\}$ is the resolution of a shiftable $(20,4;3)$ Heffter space. 
\end{ex}

\begin{ex}\label{24}
Let $V$ be the half-set of $[-24,24]$ defined by 
$$V^+=\{1,4,5,7,9,12,14,16,18,20,21,23\},\quad V^-=-\{2,3,6,8,10,11,13,15,17,19,22,24\}$$
and consider the following partitions of $V$:
\medskip\noindent
\begin{center}
\begin{tabular}{|l|c|r|c|r|c|r|c|r|c|r|c|r|}
\hline {\quad\quad\quad$\mathcal {P}_1$}    \\
\hline $\{1, -2, -13, 14\}$\\
\hline $\{-3, 4, 5, -6\}$  \\
\hline $\{7, -8, -15, 16\}$\\
\hline $\{9,-10, -11,12\}$  \\
\hline $\{-17, 18, 23, -24\}$\\
\hline $\{-19,20,21,-22\}$ \\
\hline
\end{tabular}\quad\quad\quad
\begin{tabular}{|l|c|r|c|r|c|r|c|r|c|r|c|r|}
\hline {\quad\quad\quad$\mathcal{P}_2$}    \\
\hline $\{1,-6,-15,20\}$  \\
\hline $\{-2, 7, 12,-17\}$\\
\hline $\{-3, 9, 16, -22\}$ \\
\hline $\{4, -8, -19, 23\}$\\
\hline $\{5,-10,-13,18\}$  \\
\hline $\{-11,14, 21,-24\}$\\
\hline
\end{tabular}\quad\quad\quad
\begin{tabular}{|l|c|r|c|r|c|r|c|r|c|r|c|r|}
\hline {\quad\quad\quad$\mathcal{P}_3$}    \\
\hline $\{1,-8,-11,18\}$\\
\hline $\{-2,5,16,-19\}$\\
\hline $\{-3,7,20,-24\}$ \\
\hline $\{4,-10,-15,21\}$  \\
\hline $\{-6,9,14,-17\}$\\
\hline $\{12,-13,-22,23\}$  \\
\hline
\end{tabular}
\end{center}

\medskip
It is easy to check that $\{\mathcal {P}_1,\mathcal {P}_2,\mathcal {P}_3\}$ is the resolution of a shiftable $(24,4;3)$ Heffter space. 
\end{ex}

The known results about integer Heffter arrays (see \cite{ADDY,DW}) allow us to state the following.

\begin{prop}
There exists an integer $(nk,k;2)$ Heffter space if and only if $nk\equiv 0$ or $3 \pmod 4$.
There exists a shiftable $(nk,k;2)$ Heffter space if and only if
$k$ is even and $nk\equiv 0 \pmod 4$.
\end{prop}

In this paper we begin to investigate the existence of shiftable Heffter spaces of degree $r\geq3$.
In Section 2 we present a recursive construction which, starting from any shiftable Heffter space,
leads to infinitely many shiftable Heffter spaces of the same degree. In Section 3 we present a direct 
construction of a $(16\ell^2,4\ell;3)$ Heffter space for any $\ell$, with the crucial ingredient of {\it pandiagonal magic squares}.
Combining these constructions we can say that there exists a shiftable $(16\ell^2mn,4\ell n;3)$ Heffter space for any triple
of positive integers $(\ell,m,n)$ with $m\geq n$.
We note, in particular, that this result gives infinitely many values of $v$ for which $2v+1$ is not a prime power and there exists a $(v,k;3)$ Heffter space.
Indeed, for instance, we have a $(16\ell^2,4\ell;3)$ Heffter space for any $\ell$ and it is evident that $2(16\ell^2)+1$ is not a prime power for all $\ell\not\equiv0$ (mod 3).

\section{A recursive construction for shiftable Heffter spaces}\label{sec:3}
We prove that any single shiftable Heffter space allows to obtain, recursively, infinitely
many other Heffter spaces of the same degree.

\begin{thm}\label{recursive}
If there exists a shiftable $(v,k;r)$ Heffter space and a $(w,n;r)$ space,
then there exists a shiftable $(vw,kn;r)$ Heffter space.
\end{thm}
\begin{proof}
First consider the binary operation $\star$ on the set of integers $\Z$ defined by 
$$a\star b=\begin{cases}a+bv & {\rm if \ } a\geq0;\cr a-bv & {\rm if \ } a<0.\end{cases}$$
If $A$ and $B$ are sets of integers, then $A\star B$ will denote the multiset
$\{a\star b \ : \ (a,b)\in A\times B\}$.
If ${\mathcal A}$ and ${\mathcal B}$ are collections of sets of integers, then 
we set ${\mathcal A}\star{\mathcal B}=\{A\star B \ | \ A\in{\mathcal A}; B\in{\mathcal B}\}$.

\medskip
Case 1: $r=1$.

\medskip
Let $\mathcal A$ be a shiftable $(v,k;1)$ Heffter space and let $\mathcal B$ be a $(w,n;1)$ space. 
Thus $\mathcal A$ is a shiftable $(v,k)$ Heffter system on a set $V$
and ${\mathcal B}$ is a partition of a set $W$ of size $w$ into subsets of size $n$.
We can assume, without loss of generality, that $W=[0,w-1]$.

A block of  ${\mathcal A}\star{\mathcal B}$ is of the form $A\star B$ with $A\in{\mathcal A}$ and $B\in{\mathcal B}$.
Hence we have $$(A\star B)^+=\{a+bv \ | \ a\in A^+; b\in B\},\quad (A\star B)^-=\{a-bv \ | \ a\in A^-; b\in B\}$$
so that $A\star B$ has precisely $\frac{kn}{2}$ positive elements and $\frac{kn}{2}$ negative elements.

Let $\alpha^+$ and $\alpha^-$ be the sums of the elements of $A^+$ and $A^-$, respectively. Also, let $\beta$ be the sum of the elements of $B$.
Then, considering that $A^+$ and $A^-$ have size $\frac{k}{2}$ since $\mathcal A$ is shiftable, and considering that $B$ has size $n$,
the sum of the elements of $(A\star B)^+$ is $n\alpha^+ \ + \ \frac{k}{2}\beta v$ whereas
the sum of the elements of $(A\star B)^-$ is $n\alpha^- \ - \ \frac{k}{2}\beta v$.
It follows that the sum of the elements of $A\star B$ is 
$n(\alpha^+ \ + \ \alpha^-)$ which is zero since each $A\in{\mathcal A}$ is zero-sum by assumption.

We have $$\bigcup_{A\in{\mathcal A}}A^+=V^+,\quad \bigcup_{A\in{\mathcal A}}A^-=V^-,\quad{\rm and}\quad
\bigcup_{B\in{\mathcal B}}B=[0,w-1].$$ 
Then, if $U$ is the union of all the blocks of ${\mathcal A}\star{\mathcal B}$ we can write:
$$U^+=\bigcup_{i=0}^{w-1}(V^+ \ + \ iv)\quad{\rm and} \quad U^-=\bigcup_{i=0}^{w-1}(V^- \  - \ iv).$$
Thus we see that $U$ is a half-set of $[vw,-vw]$.

 We conclude that ${\mathcal A}\star{\mathcal B}$ is a shiftable $(vw,kn)$ Heffter system.

 \medskip
Case 2: $r=2$.

\medskip
Let $\{\mathcal A,\mathcal A'\}$ be the resolution of a shiftable $(v,k;2)$ Heffter space
and let $\{\mathcal B,\mathcal B'\}$ be the resolution of a $(w,n;2)$ space.
Then both  ${\mathcal A}\star{\mathcal B}$ and  ${\mathcal A}'\star{\mathcal B}'$ are shiftable $(vw,kn)$ Heffter systems in view of the first case.
Let $A\star B$ be a block of the former system and let $A'\star B'$ be
a block of the latter. Assume that that there are two elements $x_1$, $x_2$  of $A\star B$ which also belong to $A'\star B'$. Consider for instance the case that both $x_1$ and $x_2$ are positive
(we can reason similarly in the other cases). Then we have $x_1=a_1+b_1v=a'_1+b'_1v$ and $x_2=a_2+b_2v=a'_2+b'_2v$
for suitable elements $a_1, a_2\in A^+$, $b_1, b_2\in B$, $a'_1, a'_2\in A'^+$ and $b'_1, b'_2\in B'$. 
It follows, in particular, that $a_1-a'_1$  and $a_2-a'_2$  are multiples of $v$. Considering that $A^+$ and $A'^+$ are both contained in $[1,v]$, this necessarily implies that $a_1=a'_1$ and $a_2=a'_2$.
Consequently, we also have $b_1=b'_1$ and $b_2=b'_2$. Thus $\{a_1,a_2\}\subset A\cap A'$ and $\{b_1,b_2\}\subset B\cap B'$. This is possible only for $a_1=a_2$ and $b_1=b_2$ since 
$\mathcal A$ is orthogonal to $\mathcal A'$ and $\mathcal B$ is orthogonal to $\mathcal B'$. Thus $x_1=x_2$, i.e., $A\star B$ and $A'\star B'$ have at most one common element.
We conclude that  ${\mathcal A}\star{\mathcal B}$ and  ${\mathcal A'}\star{\mathcal B}'$ are orthogonal, i.e., $\{{\mathcal A}\star{\mathcal B},{\mathcal A}^\prime\star{\mathcal B}^\prime\}$ is the resolution of a shiftable $(vw,kn;2)$ Heffter space.
 
 \medskip
Case 3: $r>2$.

\medskip
Let $\{\mathcal A_1,\dots,\mathcal A_r\}$ be the resolution of a shiftable $(v,k;r)$ Heffter space
and let $\{\mathcal B_1,\dots,\mathcal B_r\}$ be the resolution of a $(w,k;r)$ space.
In view of Case 1, $\mathcal A_i \star \mathcal B_i$ is a shiftable $(vw,kn)$ Heffter system for each $i$. Also, in view of Case 2, $\mathcal A_i \star \mathcal B_i$ is orthogonal to $\mathcal A_j \star \mathcal B_j$
for $i\neq j$.
Thus $\{\mathcal A_1 \star \mathcal B_1, \dots, \mathcal A_r \star \mathcal B_r\}$ is a set of $r$ mutually orthogonal shiftable $(vw,kn;r)$ Heffter systems, i.e., 
a shiftable $(vw,kn;r)$ Heffter space.
\end{proof}

\begin{ex}
Following the instructions in the proof of
Theorem \ref{recursive} we construct a shiftable $(40,4;3)$ Heffter space
using the shiftable $(20,4;3)$ Heffter space of Example \ref{20} as $\mathcal A$,
and the trivial $(2,1;3)$ space as $\mathcal B$. Thus the point set of $\mathcal B$ is $W=[0,1]$ and its resolution
is $\{\{0\},\{1\}\}$ repeated three times.
The point set of the resulting $(40,4;3)$ Heffter space is $U=V\star[0,1]=U^+ \ \cup \ U^-$
with $$U^+=\{2, 4, 6, 9, 10, 11, 12, 15, 17, 19, 22, 24, 26, 29, 30, 31, 32, 35, 37, 39\},$$
$$U^-=-\{1, 3, 5, 7, 8, 13, 14, 16, 18, 20, 21, 23, 25, 27, 28, 33, 34, 36, 38, 40\}$$
and its resolution $\{\mathcal {Q}_1,\mathcal{Q}_2,\mathcal{Q}_3\}$ is as follows where the blocks in bold are those
of Example \ref{20}.

\medskip\noindent
\begin{center}
\begin{tabular}{|l|c|r|c|r|c|r|c|r|c|r|c|r|}
\hline {\quad\quad\quad$\mathcal {Q}_1$}    \\
\hline $\bf\{-1, 2, 17, -18\}$\\
\hline $\{-21, 22, 37, -38\}$\\
\hline $\bf\{-3, 6, 10, -13\}$\\
\hline $\{-23, 26, 30, -33\}$\\
\hline $\bf\{4, -5, -8, 9\}$ \\
\hline $\{24, -25, -28, 29\}$ \\
\hline $\bf\{-7, 12, 15, -20\}$  \\
\hline $\{-27, 32, 35, -40\}$  \\
\hline $\bf\{ 11, -14, -16, 19\}$  \\
\hline $\{ 31, -34, -36, 39\}$  \\
\hline
\end{tabular}\quad\quad\quad
\begin{tabular}{|l|c|r|c|r|c|r|c|r|c|r|c|r|}
\hline {\quad\quad\quad$\mathcal {Q}_2$}    \\
\hline $\bf\{-1,6,9, -14\}$\\
\hline $\{-21,26,29, -34\}$\\
\hline $\bf\{2, -5, -7, 10\}$  \\
\hline $\{22, -25, -27, 30\}$  \\
\hline $\bf\{-3, 4, 19, -20\}$\\
\hline $\{-23, 24, 39, -40\}$\\
\hline $\bf\{-8, 11, 15, -18\}$ \\
\hline $\{-28, 31, 35, -38\}$ \\
\hline $\bf\{12, -13, -16, 17\}$  \\
\hline $\{32, -33, -36, 37\}$  \\
\hline
\end{tabular}\quad\quad\quad
\begin{tabular}{|l|c|r|c|r|c|r|c|r|c|r|c|r|}
\hline {\quad\quad\quad$\mathcal{Q}_3$}    \\
\hline $\bf\{-1, 10, 11, -20\}$\\
\hline $\{-21, 30, 31, -40\}$\\
\hline $\bf\{2, -8, -13, 19\}$  \\
\hline $\{22, -28, -33, 39\}$  \\
\hline $\bf\{-3, 9, 12, -18\}$ \\
\hline $\{-23, 29, 32, -38\}$ \\
\hline $\bf\{4, -7, -14, 17\}$\\
\hline $\{24, -27, -34, 37\}$\\
\hline $\bf\{-5, 6, 15, -16\}$  \\
\hline $\{-25, 26, 35, -36\}$  \\
\hline
\end{tabular}
\end{center}
\end{ex}

\begin{ex} 
Following the instructions in the proof of
Theorem \ref{recursive} we construct a shiftable $(80,8;3)$ Heffter space
using again the shiftable $(20,4;3)$ Heffter space of Example \ref{20} as $\mathcal A$,
and using the one-factorization 
$$\bigl{\{}\{0,1\},\{2,3\}\},\quad \{\{0,2\},\{1,3\}\},\quad \{\{0,3\},\{1,2\}\}\bigl\}$$
of the complete graph on $W=[0,3]$ as $(4,2;3)$ space $\mathcal B$. 
The resolution $\{\mathcal {Q}_1,\mathcal{Q}_2,\mathcal{Q}_3\}$ of the resulting $(80,8;3)$ Heffter space is as follows
where the ``sub-blocks" in bold are the blocks of Example \ref{20}.

\small
\medskip\noindent
\begin{center}
\begin{tabular}{|l|c|r|c|r|c|r|c|r|c|r|c|r|}
\hline {\quad\quad\quad$\mathcal {Q}_1$}    \\
\hline $\{{\bf-1, 2, 17, -18},-21,22,37,-38\}$\\
\hline $\{-41, 42, 57, -58,-61,62,77,-78\}$\\
\hline $\{{\bf-3, 6, 10, -13},-23,26,30,-33\}$\\
\hline $\{-43, 46, 50, -53,-63,66,70,-73\}$\\
\hline $\{{\bf4, -5, -8, 9},24,-25,-28,29\}$ \\
\hline $\{44, -45, -48, 49,64,-65,-68,69\}$ \\
\hline $\{{\bf-7, 12, 15, -20}, -27, 32, 35, -40\}$  \\
\hline $\{-47, 52, 55, -60, -67, 72, 75, -80\}$  \\
\hline $\{{\bf11, -14, -16, 19}, 31, -34, -36, 39\}$  \\
\hline $\{ 51, -54, -56, 59, 71, -74, -76, 79\}$  \\
\hline
\end{tabular}\quad\quad\quad
\begin{tabular}{|l|c|r|c|r|c|r|c|r|c|r|c|r|}
\hline {\quad\quad\quad$\mathcal {Q}_2$}    \\
\hline $\{{\bf-1,6,9, -14},  -41, 46, 49, -54\}$\\
\hline $\{-21,26,29, -34, -61, 66, 69, -74\}$\\
\hline $\{{\bf2, -5, -7, 10}, 42, -45, -47, 50\}$  \\
\hline $\{22, -25, -27, 30, 62, -65, -67, 70\}$  \\
\hline $\{{\bf-3, 4, 19, -20}, -43, 44, 59, -60\}$\\
\hline $\{-23, 24, 39, -40, -63, 64, 79, -80\}$\\
\hline $\{{\bf-8, 11, 15, -18}, -48, 51, 55, -58\}$ \\
\hline $\{-28, 31, 35, -38, -68, 71, 75, -78\}$ \\
\hline $\{{\bf12, -13, -16, 17}, 52, -53, -56, 57\}$  \\
\hline $\{32, -33, -36, 37, 72, -73, -76, 77\}$  \\
\hline
\end{tabular}
\end{center}

\begin{center}
\begin{tabular}{|l|c|r|c|r|c|r|c|r|c|r|c|r|}
\hline {\quad\quad\quad$\mathcal {Q}_3$}    \\
\hline $\{{\bf-1, 10, 11, -20}, -61, 70, 71, -80\}$\\
\hline $\{-21, 30, 31, -40, -41, 50, 51, -60\}$\\
\hline $\{{\bf2, -8, -13, 19}, 62, -68, -73, 79\}$  \\
\hline $\{22, -28, -33, 39, 42, -48, -53, 59\}$  \\
\hline $\{{\bf-3, 9, 12, -18}, -63, 69, 72, -78\}$ \\
\hline $\{-23, 29, 32, -38, -43, 49, 52, -58\}$ \\
\hline $\{{\bf4, -7, -14, 17}, 64, -67, -74, 77\}$\\
\hline $\{24, -27, -34, 37, 44, -47, -54, 57\}$\\
\hline $\{{\bf-5, 6, 15, -16}, -65, 66, 75, -76\}$  \\
\hline $\{-25, 26, 35, -36, -45, 46, 55, -56\}$  \\
\hline
\end{tabular}
\end{center}
\end{ex}

\section{Shiftable Heffter Nets of degree $3$ from Pandiagonal Magic Squares}\label{Magic}

In this section, for a given $n \times n$ array $A$, we call {\it right diagonals} of $A$ the $n$-tuples $d_0^r(A)$,  $d_1^r(A)$, \dots,  $d_{n-1}^r(A)$  
defined by $$d_j^r(A)=(a_{1,j+1},a_{2,j+2},\dots,a_{n,j+n})\quad {\rm for \ } 0\leq j\leq n-1$$
where the subscripts have to be understood modulo $n$ in  the set $[1,n]$.
Analogously, the {\it left diagonals} of $A$ are the $n$-tuples $d_0^\ell(A)$,  $d_1^\ell(A)$, \dots,  $d_{n-1}^\ell(A)$ defined by 
$$d_j^{\ell}(A)=(a_{1,j-1},a_{2,j-2},\dots,a_{n,j-n})\quad {\rm for \ } 0\leq j\leq n-1$$
with the same agreement on the subscripts as above.
Note that $d_0^r(A)$ and $d^\ell_1(A)$ are the main and the secondary diagonal of $A$, respectively.
For $n$ even, let us say that a right diagonal $d^r_j(A)$ is {\it good} if $j$ is even and that a left diagonal $d^\ell_j(A)$ is {\it good} if $j$ is odd.
The following remark is straightforward.
\begin{prop}\label{diagonals}
Let $A$ be a $n \times n$ array whose entries are pairwise distinct. Denote by $R$, $C$, $D^r$ and $D^\ell$ the sets of rows, columns, right diagonals and left diagonals of $A$.
Also, for $n$ even, denote by $G$ the set of all good diagonals of $A$. Then we have:
\begin{itemize}
\item[] for $n$ odd, $\{R,C,D^r,D^\ell\}$ is the resolution of a $(n^2,n;4)$ space;
\item[] for $n$ even, $\{R,C,D^r\}$, $\{R,C,D^\ell\}$, and $\{R,C,G\}$ are the resolutions of a $(n^2,n;3)$ space.
\end{itemize}
\end{prop}
Note that for $n$ even, we cannot say that $\{R, C, D^r, D^\ell\}$ is the resolution of a $(n^2,n;4)$ Heffter space since, for instance, 
the right diagonal $d^r_1(A)$ and the left diagonal $d^\ell_1(A)$ have the two elements $a_{\frac{n}{2},\frac{n}{2}+1}$ and $a_{n,1}$ in common.

In this section we show how a suitable class of magic squares allows to obtain a direct construction for shiftable $(n^2,n;3)$ Heffter spaces with $n\equiv 0 \pmod 4$.
From now on, two entries of a $n\times n$ array $A$ with $n$ even will be said {\it conjugates} if they are of the form $a_{i,j}$ and $a_{i+\frac{n}{2},j+\frac{n}{2}}$ for suitable indices $i$, $j$.

We recall that a {\it magic square} of order $n$ with {\it magic constant} $d$ is an $n\times n$ array 
with integer entries such that the rows, the columns, the main diagonal and the secondary diagonal sum to $d$.
It is {\it normal} when its entries
are all the integers in $[1,n^2]$.
In this case a trivial doubling counting shows that the magic constant $d$ is equal to
$\frac{n(n^2+1)}{2}$. 
A normal magic square is {\it pandiagonal} if every right diagonal and every left diagonal also sums to the magic constant $d$.
We need pandiagonal magic squares in which all pairs of conjugates elements sum up to $n^2+1$. We will call them {\it Margossian magic squares}
since we have found a method for their construction attributed to the Ottoman Engineer and Mathematician Aram Margossian (see Chapter 7 of \cite{BC}).
For the convenience of the reader we describe this method in detail below.

\begin{lem}\label{margossian}
There exists a Margossian magic square of order $n$ for every doubly-even integer $n$.
\end{lem}
\begin{proof}
Construct the {\it Graeco-Latin square} $A=(a_{i,j})$ with entries in $\Z_n\times\Z_n$ defined by

\smallskip\begin{center}
$a_{i,j}=(\frac{n}{2} i+j,i+\frac{n}{2} j)\quad 1\leq i\leq n, \ 1\leq j\leq n,$
\end{center}
where the elements of $\Z_n$ are represented with the numbers from $0$ to $n-1$.
\smallskip
Then consider the permutation $\pi$ on $\Z_n$ fixing all elements from $0$ to $\frac{n}{2}-1$ and reversing the order of the remaining 
elements $\frac{n}{2}, \frac{n}{2}+1,\dots, n-1$.
Thus, in a formula, we have
$$\pi(x)=\begin{cases}x & {\rm if} \ x\in \{0,1,\dots,\frac{n}{2}-1\}; \medskip\cr \frac{3}{2}n-1-x & {\rm if} \ x\in \{\frac{n}{2},\frac{n}{2}+1,\dots,n-1\}.\end{cases}$$
Then, the matrix $M(n)$ obtainable from $A$ by replacing each entry $(x,y)$ with $n\cdot\pi(x)+\pi(y) +1$ is a Margossian magic square.
\end{proof}

\begin{ex}\label{ex:M4}
Applying the Margossian's construction with $n=4$ we first obtain the Graeco-Latin square
$$ A=\begin{array}{|r|r|r|r|r|r|r|r|} \hline
33 & 01 & 13 & 21\\ \hline
10 & 22 & 30 & 02\\ \hline
31 & 03 & 11 & 23\\ \hline
12 & 20 & 32 & 00\\ \hline
\end{array}$$
and then we get the Margossian square
$$M(4)=\begin{array}{|r|r|r|r|r|r|r|r|} \hline
11 & 2 & 7 & 14\\ \hline
5 & 16 & 9 & 4\\ \hline
10 & 3 & 6 & 15\\ \hline
8 & 13 & 12 & 1\\ \hline
\end{array}$$
\end{ex}

\begin{ex}\label{ex:M8}
Applying the Margossian's construction with $n=8$ we first obtain the Graeco-Latin square
$$ A=\begin{array}{|r|r|r|r|r|r|r|r|} \hline
55 & 61 & 75 & 01 & 15 & 21 & 35 & 41\\ \hline
1 6 & 2 2 & 3 6 & 4 2 & 5 6 & 6 2 & 7 6 & 0 2\\ \hline
57 & 63 & 77 & 03 & 17 & 23 & 37 & 43\\ \hline
10 & 24 & 30 & 44 & 50 & 64 & 70 & 04\\ \hline
51 & 65 & 71 & 05 & 11 & 25 & 31 & 45\\ \hline
12 & 26 & 32 & 46 & 52 & 66 & 72 & 06\\ \hline
53 & 67 & 73 & 07 & 13 & 27 & 33 & 47\\ \hline
14 & 20 & 34 & 40 & 54 & 60 & 74 & 00\\ \hline
\end{array}
$$
and then we get the Margossian square
$$M(8)=\begin{array}{|r|r|r|r|r|r|r|r|} \hline
55 & 42 & 39 & 2 & 15 & 18 & 31 & 58\\ \hline
14 & 19 & 30 & 59 & 54 & 43 & 38 & 3\\ \hline
53 & 44 & 37 & 4 & 13 & 20 & 29 & 60\\ \hline
9 & 24 & 25 & 64 & 49 & 48 & 33 & 8\\ \hline
50 & 47 & 34 & 7 & 10 & 23 & 26 & 63\\ \hline
11 & 22 & 27 & 62 & 51 & 46 & 35 & 6\\ \hline
52 & 45 & 36 & 5 & 12 & 21 & 28 & 61\\ \hline
16 & 17 & 32 & 57 & 56 & 41 & 40 & 1\\ \hline
\end{array}$$
\end{ex}

Now we show how every Margossian magic square of order $n$ quickly leads to a shiftable $(3,n)$ Heffter net.

\begin{thm}\label{thm:Magic}
There exists a shiftable $(n^2,n;3)$ Heffter space for every positive integer $n\equiv0$ $($mod $4)$.
\end{thm}
\begin{proof}
Let $M=(m_{i,j})$ be the Margossian magic square of order $n$ so that the set $E$ of its entries is the interval $[1,n^2]$.
Let us distinguish two cases according to whether $\frac{n}{4}$ is even or odd.

\smallskip\noindent
\underline{Case 1}:\quad $\frac{n}{4}$ is even.

\smallskip
Let $F$ be the $\frac{n^2}{2}$-subset of $E$ consisting of all entries $m_{i,j}$ with either 
$i$ odd and $j\in [\frac{n}{4}+1,\frac{n}{2}] \ \cup \ [\frac{3n}{4}+1,n]$ or $i$ even and 
$j\in [1,\frac{n}{4}] \ \cup \ [\frac{n}{2}+1,\frac{3n}{4}]$.
For instance, if $n=8$, then $F$ consists of the ``black" entries of $M$ shown below
\scriptsize$$\begin{pmatrix}
\circ & \circ & \bullet & \bullet & \circ & \circ & \bullet & \bullet\cr 
\bullet & \bullet & \circ & \circ & \bullet & \bullet & \circ & \circ\cr
\circ & \circ & \bullet & \bullet & \circ & \circ & \bullet & \bullet\cr 
\bullet & \bullet & \circ & \circ & \bullet & \bullet & \circ & \circ\cr
\circ & \circ & \bullet & \bullet & \circ & \circ & \bullet & \bullet\cr 
\bullet & \bullet & \circ & \circ & \bullet & \bullet & \circ & \circ\cr
\circ & \circ & \bullet & \bullet & \circ & \circ & \bullet & \bullet\cr 
\bullet & \bullet & \circ & \circ & \bullet & \bullet & \circ & \circ\cr
\end{pmatrix}$$
\normalsize
It is straightforward to check that the following conditions hold:
\begin{itemize}
\item[($i$)] the conjugate of any $f\in F$ is still in $F$;\label{F1}
\item[($ii$)] each row, each column and each diagonal of $M$ has exactly $\frac{n}{2}$ elements in $F$.
\end{itemize}
Let $M'$ be the matrix obtainable from $M$ by subtracting $n^2+1$ from each of its entries belonging to $F$:
$$m'_{i,j}=\begin{cases}m_{i,j}-(n^2+1) & {\rm if} \ m_{i,j}\in F; \medskip\cr m_{i,j} & {\rm otherwise}.\end{cases}$$

By property $(i)$ the set $F$ can be partitioned into pairs $\{f,\overline f\}$ where $\overline f$ is the conjugate of $f$.
Clearly, the set $E'$ of entries of $M'$ is obtainable from $E$ by replacing each $f\in F$ with $f-(n^2+1)$.
On the other hand we have $f-(n^2+1)=-\overline{f}$ and $\overline{f}-(n^2+1)=-f$ by definition of conjugates elements,
so that $E'=(E\setminus F) \ \cup \ (-F)$.  In other words, the set $E'$ is obtainable from the set $E=[1,n^2]$ by flipping the sign of its elements belonging to $F$.
It is then obvious that $E'$ is a half-set of $[-n^2,n^2]$. 

Let $S'=\{m'_{i_1,j_1},\dots,m'_{i_n,j_n}\}$ be either a row or a column or a diagonal of $M'$ and let 
$S=\{m_{i_1,j_1},\dots,m_{i_n,j_n}\}$ be the correspondent row or column or diagonal of $M$.
By definition of $M'$, we have 
\begin{equation}\label{S'}
\sum_{k=1}^nm'_{i_k,j_k}=\sum_{k=1}^nm_{i_k,j_k}-|S\cap F|\cdot(n^2+1).
\end{equation} 
Also, we have $\sum_{k=1}^nm_{i_k,j_k}=\frac{n(n^2+1)}{2}$ since $M$ is pandiagonal
and $|S\cap F|=\frac{n}{2}$ in view of $(ii)$. We conclude that the sum in (\ref{S'}) is null, i.e., $S'$ is zero-sum.

We conclude that the sets $R$, $C$, $D^r$ and $D^\ell$ of the rows, the columns, the right diagonals and the left diagonals
of $M'$ are partitions of $E'$ into zero-sum sets of size $n$.
Then, by Proposition \ref{diagonals}, $\{R, C, D^r\}$ is the resolution of a shiftable $(n^2,n;3)$ Heffter space.

\smallskip\noindent
\underline{Case 2}:\quad $\frac{n}{4}$ is odd.

\smallskip
Let $F$ be the $\frac{n^2}{4}$-subset of $E$ consisting of all entries $m_{i,j}$ with 
$$j\in\begin{cases}[1,\frac{n}{2}] & \mbox{for either $i\leq\frac{n}{2}$ even or $i>\frac{n}{2}$ odd} \medskip\cr [\frac{n}{2}+1,n] & {\rm otherwise}\end{cases}$$
For instance, if $n=4$, then $X$ consists of the ``black" entries of $M$ shown below
\scriptsize$$\begin{pmatrix}
\circ & \circ & \bullet & \bullet\cr 
\bullet & \bullet & \circ & \circ\cr
\bullet & \bullet & \circ & \circ\cr
\circ & \circ & \bullet & \bullet\cr
\end{pmatrix}$$
\normalsize
It is straightforward to check that condition ($i$) still holds whereas condition ($ii$) should be replaced by
\begin{itemize}
\item[($ii)^\prime$] each row, each column and each good diagonal of $M$ has exactly $\frac{n}{2}$ elements in $F$.
\end{itemize}
Reasoning exactly as in Case 1, one can see the matrix obtainable from $M$ by subtracting $n^2+1$ from all the entries belonging to $F$
is a half-pandiagonal H$(n,n)$.
\end{proof}

\begin{ex}
Starting from the M$(4)$ constructed in Example \ref{ex:M4} and following the proof of Theorem \ref{thm:Magic} we get the array $M'$ below.
The three partitions of the set of entries of $M'$ into rows, columns and good diagonals of $M'$ form the resolution of a shiftable $(16,4;3)$ Heffter space.
$$M^\prime(4)=\begin{array}{|r|r|r|r|r|r|r|r|} \hline
11 & 2 & -10 & -3\\ \hline
-12 & -1 & 9 & 4\\ \hline
-7 & -14 & 6 & 15\\ \hline
8 & 13 & -5 & -16\\ \hline
\end{array}$$
\end{ex}

\begin{ex}
Take the $\M(8)$ constructed in Example \ref{ex:M8}. Following the proof of Theorem \ref{thm:Magic} we get the array $M'$ below.
The three partitions of the set of entries of $M'$ into rows, columns and right diagonals of $M'$ form the resolution of a shiftable $(64,8;3)$ Heffter space.
$$M'=\begin{array}{|r|r|r|r|r|r|r|r|} \hline
55 & 42 & -26 & -63 & 15 & 18 & -34 & -7\\ \hline
-51 & -46 & 30 & 59 & -11 & -22 & 38 & 3\\ \hline
53 & 44 & -28 & -61 & 13 & 20 & -36 & -5\\ \hline
-56 & -41 & 25 & 64 & -16 & -17 & 33 & 8\\ \hline
50 & 47 & -31 & -58 & 10 & 23 & -39 & -2\\ \hline
-54 & -43 & 27 & 62 & -14 & -19 & 35 & 6\\ \hline
52 & 45 & -29 & -60 & 12 & 21 & -37 & -4\\ \hline
-49 & -48 & 32 & 57 & -9 & -24 & 40 & 1\\ \hline
\end{array}$$
\end{ex}

\begin{rem} 
The arrays constructed in Theorem \ref{thm:Magic} are, in particular, square Heffter arrays $H(n;n)$.
Even if their existence is already known \cite{ABD}, this construction is new
and, as far as we are aware, this is the first time that a connection between magic squares and Heffter arrays
is highlighted.
\end{rem}

It is an easy exercise to prove that there exists a $(mn,n;3)$ space for every pair of integers $(m,n)$ with $m\geq n\geq 1$.
Thus, combining Theorem \ref{recursive} with Theorem \ref{margossian} we can state the following.

\begin{thm}
There exists a shiftable $(16\ell^2mn,4\ell n;3)$ Heffter space for every triple of positive integers $(\ell,m,n)$ with $m\geq n$.
\end{thm}

\section*{Acknowledgements}
The authors are partially supported by INdAM - GNSAGA.

\end{document}